\documentclass[12pt, letterpaper]{amsart}

\usepackage{amsxtra}
\usepackage{graphicx}
\usepackage{newlfont}
\usepackage{amscd}
\usepackage{hyperref}
\usepackage[german,english]{babel}
\usepackage{cancel}
\usepackage{amsmath,amsthm}
\usepackage{amsmath,amscd}
\usepackage{amssymb}
\usepackage{enumerate}
\usepackage{color}
\usepackage[vcentermath]{youngtab}
\usepackage{xypic,calc}
\usepackage[all]{xy}

 \newtheorem{thm}{Theorem}[section]

 \newtheorem{cor}[thm]{Corollary}
 \newtheorem{lem}[thm]{Lemma}
 \newtheorem{prop}[thm]{Proposition}
 
  \theoremstyle{definition}
 \newtheorem{defn}[thm]{Definition}
  
    \newtheorem{rmk}[thm]{Remark}
  \newtheorem{defn-thm}[thm]{Definition-Theorem}
 \theoremstyle{remark}

 \newtheorem{ex}[thm]{Example}

\numberwithin{equation}{section}

\numberwithin{thm}{section}

\numberwithin{table}{section}

\numberwithin{figure}{section}

\ifx\pdfoutput\undefined
  \DeclareGraphicsExtensions{.pstex, .eps}
\else
  \ifx\pdfoutput\relax
    \DeclareGraphicsExtensions{.pstex, .eps}
  \else
    \ifnum\pdfoutput>0
      \DeclareGraphicsExtensions{.pdf}
    \else
      \DeclareGraphicsExtensions{.pstex, .eps}
    \fi
  \fi
\fi

\newcommand{\ZZ}{\mathbb{Z}}

\newcommand{\CC}{\mathbb{C}}

\newcommand{\VV}{\mathcal{V}}

\newcommand{\Ind}{\text{Ind}}
\newcommand{\Res}{\text{Res}}
\newcommand{\vac}{\left| 0 \right>}

\newcommand{\ber}{\begin{red}}
\newcommand{\er}{\end{red}}

\begin{document}

\title{ Categorification of  Virasoro-Magri \\ Poisson vertex algebra}

\author{ Seok-Jin Kang$^{1}$}
\address{Department of Mathematical Sciences and Research Institute of Mathematics, Seoul National University, 599 Gwanak-Ro, Seoul 151-747, Korea}
\email{sjkang@math.snu.ac.kr}

\author{ Uhi Rinn Suh$^{2}$}
\address{Research Institute of Mathematics, Seoul National University, 599 Gwanak-Ro, Gwanak-Gu, Seoul 151-747, Korea}
\email{uhrisu1@math.snu.ac.kr}

\thanks{$^{1}$This work was supported by NRF Grant \# 2014-021261 and NRF Grant \# 2010-0010753.}

\thanks{$^{2}$This work was supported by NRF Grant \# 2014-021261.}

\maketitle

\pagestyle{plain}

\begin{abstract}
Let $\Sigma$ be the direct sum of algebra of symmetric groups $\CC\Sigma_n$, $n\in \ZZ_{\geq
0}$. We show that the Grothendieck group $K_0(\Sigma)$ of the category of
finite dimensional modules of $\Sigma$ is isomorphic to the differential
algebra of polynomials $\ZZ[\partial^n x\, |\, n\in \ZZ_\geq 0].$
Moreover, we define $m$-th products ($m\in \ZZ_{\geq 0}$) on
$ K_0(\Sigma)$ which make the algebra $ K_0(\Sigma)$ 
isomorphic to an integral form of the Virasoro-Magri  Poisson vertex algebra. Also, we
investigate relations between $K_0(\Sigma)$ and $K_0(N)$ where $K_0(N)$
is the direct sum of Grothendieck groups $K_0(N_n)$, $n\geq 0$, of finitely generated projective
$N_n$-modules. Here $N_n$ is the nil-Coxeter algebra generated by $n-1$ elements.

\end{abstract}



\section*{Introduction}

A {\it Poisson vertex algebra (PVA)} arises in mathematical physics as the underlying algebraic structure of the classical field theory. Also it is connected to other algebraic structures in mathematical physics. For example, it is a quasi-classical limit of a family of vertex algebras, which can be seen as the algebraic structure appearing in the 2-dimensional conformal field theory. Moreover, PVAs with Hamiltonian operators  are chiralizations of Poisson algebras which are related to the classical mechanics. Here the chiralization implies a Poisson algebra of finite algebraic dimension is connected to a PVA of  infinite algebraic dimension.
(See \cite{K}.)

In \cite{Z}, Zhu constructed the maps called the Zhu maps which relates an associative algebra $Zhu_H(V)$ to a vertex algebra $V$ with a Hamiltonian operator $H$. He proved that there is a one-to-one correspondence between irreducible positive energy modules over a vertex algebra $V$ and irreducible modules over the Zhu algebra $Zhu_H(V)$. Analogously, De Sole and Kac \cite{DK} took the quasi-classical limit of the Zhu map and obtained the  (classical) Zhu map which associates a Poisson algebra  $Zhu_H(\VV)$ to a PVA $\VV$ with a Hamiltonian operator $H$. Hence, from the point of view of representation theory, the Zhu map is a reasonable finitization map from a PVA to  a Poisson algebra. Thus we consider a PVA as a chiralization of its Zhu algebra. The simplest example of Zhu algebras is  $Zhu_H(\VV)=\CC[x]$, where $\VV$ is the Virasoro-Magri PVA and $H=L_0$ for an energy momentum field $L \in \VV.$

On the other hand, Khovanov \cite{Kho} showed that the direct sum of Grothendieck groups of finitely generated projective modules of the nil-Coxeter algebras is isomorphic to the polynomial algebra $\ZZ[x].$ More precisely, let $N_n$, $n\in \ZZ_{\geq 0}$, be the nil-Coxeter algebra generated by $n-1$ elements and let $K_0(N_n)$ be the Grothendieck group of the category $N_n$-pmod of finitely generated projective modules over $N_n$.   Then $K_0(N)=\bigoplus_{n\geq 0} K_0(N_n)$ is isomorphic to $\ZZ[x]$. Moreover, he showed that the induction and restriction functors on the direct sum $\bigoplus_{n\geq 0} N_n$-mod of categories of finitely generated left $N_n$-modules  categorify the polynomial representation of the Weyl group. (See Section \ref{Sec:SGNA}.)

Our natural question is how to categorify Virasoro-Magri PVA $\VV$, a chiralization of the polynomial algebra. Let $\Sigma_n$ be the symmetric group on $n$ letters and set $\Sigma=\bigoplus_{n\geq 0} \CC \Sigma_n$. We denote by $K_0(\Sigma)$ the Grothendieck group of the category of finitely generated projective modules over $\Sigma$. The main result of this paper shows that $K_0(\Sigma)$ is isomorphic to $\VV_\ZZ$ as PVA, where $\VV_\ZZ$ is the integral form of $\VV$ defined in Remark \ref{VM_Z} (2) (Theorem \ref{Thm:3.8_1106}).

Recall that the nil-Coxeter algebra is a degenerate homogenous version of the symmetric group algebra. Another main result of this paper is difference of these two algebras becomes manifest through the Zhu map between Virasoro-Magri PVA and the polynomial algebra $\CC[x].$

Our paper is organized as follows. In Section \ref{Sec:PVA}, we review the notion of vertex algebras and the state-field correspondence. In Section \ref{Sec:PVA_1}, we give an explicit description of a PVA as the quasi-classical limit of a family of vertex algebras. We also explain the Zhu maps and Poisson algebra structures on Zhu algebras. In Section \ref{Sec:SGNA}, we review the representation theory of symmetric groups and nil-Coxeter algebras. Our main results are included in Section \ref{Sec:VMPVA} and Section \ref{Sec:app}. In Section \ref{Sec:VMPVA}, we prove that the category of finitely generated projective modules over $\Sigma$ provides a categorification of Virasoro-Magri PVA. In Section \ref{Sec:app}, we describe the relation between $K_0(N)$ and $K_0(\Sigma)$. The PVA structure on $K_0(\Sigma)_\CC$ induces a Poisson algebra structure and a differential algebra structure on $K_0(N)_\CC$, and the differential algebra structure on $K_0(\Sigma)$ induces that of $K_0(N)$. Moreover, using the state-field correspondence, we give a realization of the quantization of $K_0(\Sigma)$ in terms of linear operators on $K_0(N)$.

\vskip 5mm

\section{Vertex algebras} \label{Sec:PVA}

A main ingredient of this paper is the notion of  Poisson vertex algebras, which can be understood as  quasi-classical limits of family of vertex algebras.

\begin{defn} \cite{K}  Let $V$ be a vector space over  $\CC$. 
\begin{enumerate}[(i)]
\item A {\it quantum field} on $V$ is
\[\Phi(z)= \sum_{n\in \ZZ} \Phi_{(n)} z^{-n-1},\]
where $\Phi_{(n)} \in \text{End} V$ and $\Phi_{(n)} v=0$ for all but finitely many $n>0$.

\item Let $\vac$ be an element in $V$,  $\partial$ be an endomorphism of $V$
and  $\mathcal{F}$ be a set of quantum fields.
A quadruple $(V, \vac, \partial, \mathcal{F})$  is a {\it pre-vertex algebra} if it satisfies:
\begin{enumerate}
\item (vaccum) $\partial \vac =0;$
\item (translational covariance) $[\partial, \Phi(z)]= \frac{d}{dz} \Phi(z);$
\item (completeness) $\Phi_{(n_s)}^{\alpha_s} \cdots \Phi_{(n_1)}^{\alpha_1}  \vac$ spans $V$, where $\Phi^{\alpha_i}\in \mathcal{F}$, $n_i\in \ZZ$, $i=1, \cdots, s;$
\item (locality) for any $\Phi^\alpha, \Phi^\beta \in \mathcal{F}$,
there exists a non-negative integer $N_{\alpha, \beta}$ such that
$(z-w)^{N_{\alpha, \beta}} [\Phi^\alpha(z), \Phi^\beta(w)]=0$.
\end{enumerate}
\end{enumerate}
\end{defn}

For each pre-vertex algebra $(V, \vac, \partial, \mathcal{F})$, let
$\widetilde{\mathcal{F}}$ be the set of all translational covariant
quantum fields on $V$ and let $\overline{\mathcal{F}}$ be
the maximal subset of $\widetilde{\mathcal{F}}$
containing $\mathcal{F}$ such that $\overline{\mathcal{F}}$
satisfies the locality condition.

\begin{prop} \cite{K} \hfill

{\rm (a)}  There is a bijective linear map
\[s: V \to \overline{\mathcal{F}}, \qquad a \mapsto a(z),\]
 whose inverse is given by $$a= a(z)\vac|_{z=0} =
a_{(-1)} \vac.$$

{\rm (b)} We have
\begin{equation}\label{eq:correspondence}
\begin{aligned}
&  s(\vac) = \text{id}_{V}, \ \ s(\partial^j a)= \partial_{z}^{j}
a(z), \ \ \text{where}
\ \partial_{z} = \frac{d}{dz}, \\
& \partial^j \, a=\partial_z^j \, a(z)\vac |_{z=0} = 
j! \,  a_{(-j-1)} \vac \ \text{ for } \ j\in \ZZ_{\geq 0}.
\end{aligned}
\end{equation}
\end{prop}
The map $s$ is called the {\it state-field correspondence}.

Since the set of quantum fields is not closed under multiplication,
we use  the {\it normally ordered product }  $: \ \ :$. The normally
ordered product of two quantum fields $a(z)$ and $b(z)$ in
$\overline{\mathcal{F}}$ is defined by
\[(a(z), b(z)) \mapsto : a(z) b(z): \ \ =\  a(z)_+b(z) +b(z)a(z)_-,\]
where $$a(z)_+= \sum_{n<0} a_{(n)} z^{-n-1} \ \ \text{and} \ \
a(z)_-= \sum_{n\geq0} a_{(n)} z^{-n-1}.$$ Then $:a(z)b(z):$ is again
a quantum field in $\overline{\mathcal{F}}.$  The corresponding
state of the normally ordered product between $a(z)$ and $b(z)$ is
\[ s^{-1}(:a(z) b(z):)= a_{(-1)}b\]
and we denote  $a_{(-1)} b$ by $:a b:$. 

\begin{defn} \cite{K}
A {\it vertex algebra} is a quadruple $(V, \vac, \partial, \overline{\mathcal{F}})$, where  $(V, \vac, \partial, \mathcal{F})$  is a pre-vertex algebra and $\overline{\mathcal{F}}$ is obtained from $\mathcal{F}$ by the state-field correspondence $s:\mathcal{V} \to \overline{\mathcal{F}}$.
\end{defn}

One of the most important examples of a vertex algebra arises from
Virasoro algebra. Recall that Virasoro algebra is the Lie algebra $
Vir = \bigoplus_{n\in \ZZ} \CC L_n \oplus \CC C $ endowed with the
bracket
\[ [L_m, L_n] = (m-n) L_{m+n} +\delta_{m+n, 0} \frac{1}{12}(m^3-m) C, \quad [Vir, C]=0.\]

\begin{rmk} \label{Vir}
In $Vir/\CC C$,  we may identify $L_m$ with the operator $x^{-m+1}
\frac{d}{dx}: \CC[x] \to\CC[x].$ Indeed, we have $$[x^{-m+1}
\frac{d}{dx}, x^{-n+1} \frac{d}{dx}]= (m-n) x^{-m-n+1} \frac{d}{dx}.$$
\end{rmk}

Let $U=U(Vir)$ be the universal enveloping algebra of $Vir$ and let
$I$ be the left ideal of $U$ generated by $(Vir)_{\ge -1} :=
\bigoplus_{j \ge -1} \CC L_{j}$. Set $V=U/I$. Using the Leibniz
rule, we can define a linear operator $\partial$ on $V$ by
$$L_{-n} \mapsto (n-1)L_{-n-1} \quad (n \ge 2).$$
Take ${\mathcal F}=\{C, \ L(z) = \sum_{n \ge 2} L_{-n} z^{n-2} \}$
to be the set of quantum fields on $V$, where the commutation
relations on $\mathcal{F}$ are induced by those of $Vir$. Then the $(V,
1, \partial, \mathcal{F})$ is a pre-vertex algebra. For any $c \in
\CC$, $V_{c}:=V\big/(C-c)V$ has a vertex algebra structure induced
from $V$. By the completeness, $V_{c}$ is spanned by
\[ L_{-j_1-2}\cdots L_{j_s-2} \vac, \ \ \text{where} \ s\ge 0, \ j_1 \geq \cdots \geq j_s \geq 0.\]
The vertex algebra $V_{c}$ is called the {\it universal Virasoro
vertex algebra with central charge $c$} \cite{K}.

\vskip 2mm


\vskip 2mm

For the universal Virasoro vertex algebra, 
the state-field correspondence is
\[ s(L_{j_1-2} \cdots L_{-j_s-2}\vac )= : \partial^{(j_1)}_z L(z) \cdots \partial^{(j_s)}_z L(z):
\ \ \text{where} \ \ \partial^{(j)} = \frac{1}{j!} \partial^j \]
and
\[: a^1(z) a^2(z)\cdots a^s(z):  =
(:a^1(z)(:a^2(z): \cdots (: a^{s-1}(z) a^s(z):)\cdots:):) \] for
$a^i \in V$, $1\le i \le s.$


\vskip 3mm

We now introduce an alternative definition of vertex algebras which is 
equivalent to the previous one. We first recall the definition of
differential algebras and Lie conformal algebras \cite{K}.


\begin{defn} \hfill
\begin{enumerate}[(i)]
\item A unital associative algebra $A$ is called a {\it differential algebra}
if it has a linear operator $\partial:A\to A$ such that $\partial
(a\cdot b)= \partial a\cdot b+ a \cdot \partial b.$
\item A {\it Lie conformal algebra}  $(R, [ \, _\lambda\, ] ) $ is a $\CC[\partial]$-module
endowed with a linear map
\[  [ \, _\lambda\, ]: R \otimes R \to R[\lambda] \]
for each formal variable $\lambda$, called the $\lambda$-bracket,
satisfying the following properties:
\begin{enumerate}
\item (sesqui-linearity) $ [\partial a_\lambda b]=-\lambda[a_\lambda b], \quad [a_\lambda \partial b]=(\lambda+\partial)[a_\lambda b],$

\item (skew-symmetry) $[b_\lambda a]= -[a_{-\lambda-\partial} b],$

\item (Jacobi identity) $[a_\lambda[b_\mu c]]= [[a_\lambda b]_{\lambda+\mu}c] +[b_\mu[a_\lambda c]].$
\end{enumerate}

\end{enumerate}
\end{defn}

Write $$[a_\lambda b]= \sum_{n\in \ZZ_{\geq 0}} \frac{\lambda^n}{n!}
c_{n} \in R[\lambda]$$ and define $a_{(n)} b = c_{n}$ for $n \in
\ZZ_{\ge 0}$. So a Lie conformal algebra is endowed with infinitely
many products $a_{(n)} b$, the $n$-th product of $a$ and $b$,
indexed by $\ZZ_{\ge 0}$.

On a vertex algebra $V$, we consider the $\lambda$-bracket $[\, _\lambda \, ]$ and the normally ordered product $: \, \, :$ such that
\[ [a_\lambda b]= \Res_{z}e^{\lambda z} a(z) b = \sum_{n\geq 0} \frac{\lambda^n}{n!} (a_{(n)}b)\]
and $: ab:=a_{(-1)}b.$  Since we have
\[ :\frac{(\partial^j a)  b}{j!} := a_{(-j-1)} b,\]
the $\lambda$-bracket determines all non-negative $n$-th products on
$V$ and the normally ordered product determines all  negative $n$-th
products on $V$. In this way, we get an alternative definition of a
vertex algebra via $\lambda$-brackets and normally ordered products.

\begin{defn} \cite{BK}
The quintuple $(V, \vac, \partial, [\, _\lambda \, ], : \, \, :)$ is a {\it vertex algebra} if 
\begin{enumerate}
\item $(V, \partial,  [\, _\lambda \, ])$ is a Lie conformal algebra,
\item $(V, \vac, \partial, : \, \, :)$ is a differential algebra with strong

quasi-commutativity; i.e.,
\[ :a\,(:bc:):-:b\,(:ac:):= \sum_{n\geq 0} : \left(\frac{(-\partial)^{n+1}}{(n+1)!} a_{(n)}b \right) \, c: \ \ \text{for} \ a,b,c \in V,\]

\item the normally ordered product and the $\lambda$-bracket are related by non-commutative Wick
formula; i.e.,
\[ [a_\lambda (:bc:)]= :[a_\lambda b] c:+:b[a_\lambda c]: + \int_0^\lambda [[a_\lambda b]_\mu c]] d\mu\ \text{ for } a,b,c \in V,\]
where $\int_0^\lambda A\mu^n d\mu= \frac{1}{n+1} \lambda^{n+1} A$ \,
for $A \in V[\lambda].$
\end{enumerate}
\end{defn}

\vskip 2mm

\begin{ex} \hfill
\begin{enumerate}
\item
Let $R=\CC[\partial]\otimes L \oplus \CC C$ be a
$\CC[\partial]$-module such that $\partial C=0$ and define
\[ [L_\lambda L]=(\partial +2\lambda) L + C \lambda^3, \quad [C_\lambda R]=0.\]
Then it can be extended to a $\lambda$-bracket on $R$ by
sesqui-linearity and we obtain a Lie conformal algebra structure on
$R$, the {\it Virasoro Lie conformal algebra}.

\item Let $V$  be the vector space spanned by
$$:\partial^{(j_1)}L \cdots \partial^{(j_s)}L:  \ \ \text{for} \ s \ge 0, \,  j_1, \cdots,j_s
\in\ZZ_{\geq 0}$$ and define the $\lambda$-bracket by
$$[L_{\lambda} L]= (\partial+2\lambda)L+ c \lambda^3 \ \ \text{for some} \ \, c\in
\CC.$$ Then $V$ is the universal Virasoro vertex algebra with
central charge $c$.
\end{enumerate}

\end{ex}

\vskip 5mm

\section{Poisson vertex algebras} \label{Sec:PVA_1}

Let $(V_\epsilon, \vac_\epsilon,
\partial_\epsilon,  [\, _\lambda \, ]_\epsilon, : \, \, :_\epsilon)$
be a family of vertex algebras such that (i) $V_\epsilon$ is a free
$\CC[\epsilon]$-module, \,(ii) $[V_\epsilon, V_\epsilon] \subset
\epsilon V_\epsilon[\lambda].$ We review several notions introduced in \cite{K}.

\begin{defn} \label{def:quasi-classical limit}
The {\it quasi-classical limit} of $(V_\epsilon, \vac_\epsilon,
\partial_\epsilon,  [\, _\lambda \, ]_\epsilon, : \, \, :_\epsilon)$
is the quintuple  $(\VV, 1, \partial,  \{\, _\lambda \,\}, \, \cdot
\, )$, where
\begin{enumerate}
\item $\VV = V_{\epsilon} \big/ \epsilon V_{\epsilon}$,
\item $1 = \vac + \epsilon V_{\epsilon}$,
\item $\partial$ and $\cdot$ are induced by $\partial_{\epsilon}$
and $:\, \, :_{\epsilon}$,
\item $\{(a+\epsilon V_{\epsilon})_{\lambda} (b+ \epsilon
V_{\epsilon}) \} = [a _{\lambda} b] + \epsilon V_{\epsilon}$ for
$a,b \in V_{\epsilon}$.
\end{enumerate}
\end{defn}

\begin{defn}
The quintuple $(\VV, 1,  \partial, \{\, _\lambda \, \}, \cdot)$ is
called a {\it Poisson vertex algebra} (or {\it PVA} for brevity) if
it satisfies the following conditions:
\begin{enumerate}
\item $(\VV, \partial, 1, \cdot)$ is a commutative  differential algebra.
\item $(\VV, \partial, \{\, _\lambda \, \})$ is a Lie conformal algebra.
\item The Leibniz rule $\{a_\lambda bc\}= b\{a_\lambda c\}+ c\{a_\lambda b\}$ holds for all $a,b,c\in \VV$.
\end{enumerate}
\end{defn}

One can check that the quasi-classical limit $(\VV, 1, \partial,  \{\, _\lambda \,\}, \, \cdot
\, )$ of $(V_\epsilon, \vac_\epsilon,
\partial_\epsilon,  [\, _\lambda \, ]_\epsilon, : \, \, :_\epsilon)$ is a Poisson vertex algebra.

Let $V$ be the universal Virasoro vertex algebra with central charge
$c \in \CC$ and let $V_{\epsilon} = V[\epsilon]$ be the family of
vertex algebras endowed with the $\lambda$-bracket
$$[L_{\lambda}L]_{\epsilon} = \epsilon[L_{\lambda} L].$$
We now introduce the main subject of this paper.

\begin{defn}\label{def:Virasoro-Magri PVA}
The classical limit $\VV$ of the family of universal Virasoro vertex
algebras $V_{\epsilon}$ is called the {\it Virasoro-Magri Poisson
vertex algebra with central charge $c$}.
\end{defn}

By the definition, we have $\VV = \CC[\partial^{n} L \mid n \in
\ZZ_{\ge 0}]$ and
$$\{L_{\lambda}L \} = (\partial + 2 \lambda) L + c \lambda^3,$$
which is extended to the $\lambda$-bracket on $\VV$ by the
sesqui-linearity and the Leibniz rule.

\vskip 2mm

In general, we have the following {\it master's formula}.

\begin{prop} \cite{BDK}
Let $I$ be an index set and let $\VV=\CC[ u_i^{(n)}| n \in \ZZ_{\geq
0}, i \in I]$ be the algebra of differential polynomials where $
u_i^{(n)}=\partial^n u_i$ $(n \ge 0, i \in I)$.

If $\VV$ is a Poisson vertex algebra, then the following formula
holds:
\[ \{f_\lambda g\}= \sum_{i,j\in I,\, m,n\in \ZZ_{\geq 0}}
\frac{\partial g}{\partial u_j^{(n)}}(\lambda+ \partial)^n \{u_{i\,
\lambda+\partial } u_j\}_{\to} (-\lambda-\partial)^m \frac{\partial
f}{\partial u_i^{(m)}} \ \ \text{for} \ f,g\in \VV, \] where $\{\
_{\lambda+\partial}\ \}_\to$ means the differential $\partial$
arising from the bracket acts on the right; i.e, 
\[ \textstyle \{ a_{\lambda+\partial} b\}_{\to}C = \sum_{n\geq 0} a_{(n)}b \frac{(\lambda+\partial)^n}{n!} C\neq\sum_{n\geq 0} \frac{(\lambda+\partial)^n}{n!}  a_{(n)}b\,  C, \quad a,b,C\in \mathcal{V}. \]
\end{prop}

\begin{proof}
It directly follows from the sesqui-linearity and the Leibniz rule.
\end{proof}

\begin{ex} \label{VM-lambda}
Let $\VV$ be the Virasoro-Magri Poisson vertex algebra. Then we have
\begin{equation*}
\begin{aligned}
 & \{ \partial^m L_\lambda \partial^n L\} = ( \partial+\lambda)^n \{L_{\lambda+\partial} L\}_\to (-\partial-\lambda)^m
 =  ( \partial+\lambda)^n \{L_{\lambda} L\} (-\lambda)^m \\
 & = \sum_{i=0}^n { n \choose i } \lambda^i \partial^{n-i} (\partial +2\lambda) L (-\lambda)^m + \lambda^{n+m+3}(-1)^m c.
 \end{aligned}
\end{equation*}
Hence the $j$-th products $(j\in \ZZ_{\geq 0})$ of $\partial^m L$
and $\partial^n L$ are given by
\begin{equation*}
(\partial^{m}L)_{(j)}(\partial^{n}L) = \begin{cases} 2(m+n+1)!\, L \
\ & \text{for} \ j=m+n+1, \\
(m+n+3)!\, c \ \ & \text{for} \ j=m+n+3, \\
j! \, \left( {n \choose j-m}+2{n\choose
j-m-1}\right)\partial^{n-j+1}L \ \ & \text{for} \ j=m+1, \ldots,
m+n, \\
m! \, \partial^{n+1} L \ \ & \text{for} \ j=m, \\
0 \ \ & \text{otherwise}.
\end{cases}
\end{equation*}

In general, if we set $l_n:= \partial^n L$,  then
\begin{equation} \label{VMgen}
\begin{aligned}
&\textstyle{  \left\{  \prod_{i=1}^s l_{m_i} \, _\lambda\, \prod_{j=1}^t  l_{n_j} \right\} }  = \sum_{\substack{1 \le p \le s \\ 1 \le q \le t}} \,
\left(\prod_{j \neq q } \,  l_{n_j}\right)
(\lambda+\partial)^{n_j} \{L_{\lambda+\partial} L\}
(-\lambda-\partial)^{m_i} \left(\prod_{i\neq p} \, l_{m_i}\right).
\end{aligned}
\end{equation}
By a direct calculation, one can verify that
$$(l_{m_1} \cdots l_{m_s})_{(n)} (l_{n_1} \cdots l_{n_t}) = n! \,
c_{n},$$ where $c_{n} \in \VV$ is the coefficient of $\lambda^n$ in
\eqref{VMgen}.
\end{ex}

\begin{rmk} \label{VM_Z}  \ 
\begin{enumerate}
\item
Let $\VV_\ZZ$ be an integral form of the Virasoro-Magri PVA $\VV$. If $\{ \VV_\ZZ \, _\lambda \VV_\ZZ\} \subset \VV_\ZZ[\lambda]$, then we say that $\lambda$-bracket is well-defined on  $\VV_\ZZ$.
\item
If the central charge $c\in \ZZ$, then the $\lambda$-bracket is well-defined on the integral form $\VV_\ZZ= \ZZ[\partial^n L| \, n\in\ZZ_{\geq 0}]$.
\end{enumerate}
\end{rmk}

\vskip 3mm

In the rest of this section, we will investigate the relation
between Poisson algebras and Poisson vertex algebras. For this
purpose, we first introduce the notion of Hamiltonian operators on
Poisson vertex algebras.

\begin{defn}
Let $\VV$ be a PVA. A {\it Hamiltonian operator} on $\VV$ is a
diagonalizable linear operator  such that
\[ H(a_{(n)} b)= (Ha)_{(n)} b+ a_{(n)}Hb-(n+1) a_{(n)}b \ \ \text{for all} \ a,b\in \VV. \]
\end{defn}
For an eigenvector $a$ of $H$, we denote its eigenvalue by
$\Delta_a$ and we call it the {\it conformal weight} of $a.$

A main source of a Hamiltonian operator is an energy momentum field
$L$; i.e., $L$ satisfies
\begin{enumerate}
\item $\{L_\lambda L\}= (\partial +2\lambda) L + \frac{c}{12} \lambda^3,$
\item $L_{-1}= \partial,$
\item $L_{0}$ is a diagonalizable operator.
\end{enumerate}
Then $L_{0}$ is a Hamiltonian operator. (Here, by convention, we
denote $L_{n}=L_{(n+1)}$ $(n \in \ZZ)$.

The Poisson vertex algebras with Hamiltonian operators can be
considered as chiralizations of Poisson algebras. In other words,
the $\lambda$-brackets and the multiplications of these Poisson
vertex algebras induce Poisson brackets and the multiplications of
corresponding Poisson algebras.

Let $\VV$ be a Poisson vertex algebra with a Hamiltonian operator
$H$. Consider the vector space $\VV_\hbar:= \VV[\hbar]$ endowed with
the $\CC[\hbar]$-linear commutative associative product $\cdot$
induced from $\VV$. We define the $\CC[\hbar]$-bilinear
$\hbar$-bracket on $\VV_{\hbar}$ by
\[\{a , b\}_{\hbar}= \sum_{j\in \ZZ_{\geq 0}} {\Delta_a -1\choose j } \hbar^j a_{(j)}b
\ \ \text{for} \ a,b\in \VV.\]

Let $J_{\hbar}$ be the $\CC[\hbar]$-submodule of $\VV_{\hbar}$
generated by
\[ (\partial + \hbar H) a \cdot b  \ \ (a,b \in \VV).\]
Then we have
\[ \partial^{(n)} a \equiv \hbar^n { -\Delta_a \choose n} a \qquad \text{ mod} J_\hbar \]
and the following theorem holds.

\begin{thm} \cite{DK}
The submodule $J_\hbar$ is a two-sided ideal of $\VV_\hbar$ with respect to the commutative associative product $\cdot$ and the $\hbar$-bracket $\{ \, , \, \}_\hbar$.
\end{thm}

Hence $\VV_\hbar/J_\hbar$ is a Poisson algebra endowed with the
product $\cdot$ and the bracket $\{\, , \, \}_\hbar.$ By
specializing $\hbar=c$ for some $c\in \CC$, we get a Poisson algebra
$\VV/J_{\hbar=c}.$ In particular, when $c=1$, the algebra is called
the $H$-twisted Zhu algebra.

\begin{defn} \cite{DK}
Let $\VV$ be a Poisson vertex algebra with a Hamiltonian operator
$H$. Then the Poisson algebra $\text{Zhu}_{H}(\VV):= \VV/J_{\hbar=1}$
is called the {\it $H$-twisted Zhu algebra of $\VV$}.
\end{defn}

As can be seen in the following proposition, the Zhu algebras have a
simple description under some reasonable restrictions.

\begin{prop} \label{Prop:zhu} \cite{DK}
Let $R$ be a Lie conformal algebra (or a central extension of a Lie
conformal algebra) and let $S(R)$ be the Poisson vertex algebra
endowed with the $\lambda$-bracket which is defined by that of $R$
and the Leibniz rule.

If $\VV=S(R)$ has a Hamiltonian operator $H$, then $\text{Zhu}_H \VV
\simeq S(R/\partial R)$ as Poisson algebras, where the Poisson
bracket on $S(R/\partial R)$ is defined by
\[ \{ a+\partial R ,b+\partial R\}=\{a_\lambda b\}|_{\lambda=0}+\partial R \ \ \text{for} \ a,b \in R\]
and the Leibniz rule.
\end{prop}

\begin{rmk} Let $P$ be a Poisson algebra and let $P_\ZZ$ be an integral form of $P$. If $\{P_\ZZ, P_\ZZ\} \subset P_\ZZ$, then we say that the Poisson bracket on $P_\ZZ$ is well-defined.  
\end{rmk}

\vskip 5mm

\section{Symmetric groups and nil-Coxeter algebras} \label{Sec:SGNA}

In this section, we briefly review the representation theories of
symmetric groups and nil-Coxeter algebras. We first recall the
definition of Young diagrams and Young tableaux.

\begin{defn} Let $n$ be a non-negative integer.
\begin{enumerate}
\item A {\it partition} of $n$ is a sequence $\lambda = (\lambda_1,
\lambda_2,  \ldots, \lambda_l)$ of positive integers such that
$\lambda_1 \ge \lambda_2 \ge \cdots \ge \lambda_l >0$ and $\lambda_1
+ \lambda_2 + \cdots + \lambda_l = n$. When $\lambda$ is a partition
of $n$, we write $\lambda \vdash n$.

\item A {\it  Young diagram} $Y^{\lambda}$ for a partition $\lambda = (\lambda_1,
\lambda_2,  \ldots, \lambda_l) \vdash n$ is a collection of $n$
boxes arranged in left-justified rows with $\lambda_i$ boxes in the
$i$-th row $(i=1, \ldots, l)$. We often identify the partition
$\lambda$ with its Young diagram $Y^{\lambda}$.

\item A {\it standard Young tableau of shape $\lambda$} is an
assignment of $1, 2, \ldots, n$ to each box of $Y^{\lambda}$ such
that the entries in each row and column are strictly increasing from
left to right and from top to bottom.
\end{enumerate}
\end{defn}







Let $\Sigma_{n}$ be the symmetric group on $n$ letters and let $\CC
\Sigma_{n}$ be its group algebra. For each partition $\lambda \vdash
n$, we denote by $S^{\lambda}$ the $\CC$-vector space with basis
consisting of standard tableaux of shape $\lambda$. Then
$S^{\lambda}$ is a simple $\Sigma_{n}$-module, called the {\it
Specht module}. Moreover, we have the following theorem (cf. \cite{F, Sa}).

\begin{thm} \label{Thm:Specht}
Two Specht modules $S^\mu$ and $S^{\nu}$ are isomorphic if and only if $\mu=\nu.$
Moreover, the set of Specht modules $\{S^\mu\, |\, \mu \vdash n\}$ is a complete list of mutually non-isomorphic  irreducible representations of $\Sigma_n$.
\end{thm}

We denote by $ \Sigma_{n} \text{-mod}$ the category of finite
dimensional $ \Sigma_{n}$-modules. Consider the induction and
restriction functors given below.
\begin{equation*}
\begin{aligned}
& \text{Ind}_{n}^{n+1}:  \Sigma_{n}\text{-mod} \rightarrow 
\Sigma_{n+1}\text{-mod}, \quad M \mapsto \CC \Sigma_{n+1}
\otimes_{\CC \Sigma_{n}} M, \\
& \text{Res}_{n}^{n+1}:  \Sigma_{n+1}\text{-mod} \rightarrow 
\Sigma_{n} \text{-mod}, \quad N \mapsto \CC \Sigma_{n+1}
\otimes_{\CC \Sigma_{n+1}} N.
\end{aligned}
\end{equation*}

These functors can be visualized by the following proposition (cf.
\cite{F, Sa}).

\begin{prop}[Branching rules] \label{Thm:branching}
For each partition $\lambda \vdash n$, the induction and restriction
functors yield the following decompositions:
$$\text{Ind}_{n}^{n+1} (S^{\lambda}) = \bigoplus_{\Box} S^{\lambda
\leftarrow \Box}, \quad \text{Res}_{n}^{n+1} (S^{\lambda}) =
\bigoplus_{\Box} S^{\lambda \rightarrow \Box},$$ where $\lambda
\leftarrow \Box$  denotes a Young diagram with $(n+1)$-boxes
obtained by adding a box to $\lambda$ and $\lambda \rightarrow \Box$
is a Young diagram with $(n-1)$-boxes obtained by removing a box
from $\lambda$. 
\end{prop}



\subsection{The representation theory of nil-Coxeter algebras} \label{Subsec:nil}

We refer to the paper of Khovanov \cite{Kho} and the survey paper by Savage \cite{S} for the detailed proof of propositions and theorems in Section \ref{Subsec:nil}.
Recall the symmetric group $\Sigma_n$ is generated by simple transpositions $s_i$, $i=1, \cdots, n-1$ with defining relations 
\begin{equation} \label{Eqn:Sn}
\begin{aligned}
&  s_i^2= id \qquad \text{ for }  i=1, \cdots, n-1, \\
& s_i s_j = s_j s_i \qquad \text{ for } i,j=1, \cdots, n-1 \text{ such that } |i-j|\geq 2, \\
& s_i s_{i+1} s_i = s_{i+1} s_i s_{i+1}\qquad \text{ for } i=1, \cdots ,n-2.
\end{aligned}
\end{equation}

If we replace the first relation by $s_i^2=0$, we get another algebra $N_n$ called the nil-Coxeter algebra.

\begin{defn}
A nil-Coxeter algebra $N_n$ is generated by $n_i$, $i=2, \cdots, n-1$ with defining relations
\begin{equation} \label{Eqn:Nn}
\begin{aligned}
&  n_i^2= 0 \qquad \text{ for }  i=1, \cdots, n-1, \\
& n_i n_j = n_j n_i \qquad \text{ for } i,j=1, \cdots n-1 \text{ such that } |i-j|\geq 2, \\
& n_i n_{i+1} n_i = n_{i+1} n_i n_{i+1}\qquad \text{ for } i=1, \cdots n-2.
\end{aligned}
\end{equation}
For $n=0,1$, we let $N_n=\CC.$
\end{defn}

Since the relations of the algebra $N_n$ is homogeneous, it has a $\ZZ_{\geq 0}$-grading decomposition $N_n= \bigoplus_{k\geq 0} N_n(k)$ by letting $\deg(n_i)=1$, $i=1, \cdots, n-1.$ If we consider an ideal $I_n$  generated by $n_i$, $i=1, \cdots, n-1$, then $I_n= \bigoplus_{k>0} N_n(k)$ and $I_n$ is the maximal ideal of $N_n.$

\begin{prop} \label{Prop:2.11_1110} \cite{Kho,S}
The maximal ideal $I_n$ defined above satisfies the following properties:
\begin{enumerate}
\item $I_n$ is the unique maximal ideal of $N_n.$
\item There is a unique simple module $L_n$ which is isomorphic to $N_n/I_n.$
\end{enumerate}
\end{prop}

Let $N_n$-mod be the category of finitely generated left $N_n$-modules and let $G_0(N_n)$ be its Grothendieck group. Denote by $[M]$ the isomorphism class of $M\in N_n$-mod. Then  by Proposition \ref{Prop:2.11_1110},  we have
\[ G_0(N_n)= \ZZ[L_n] \quad \text{ and } \quad [N_n]=n!\, [L_n]. \]

On the other hand,  let $N_n$-pmod be the category of finitely generated projective $N_n$-modules
and let $K_0(N_n)$ be its Grothendieck group. 

\begin{prop} \cite{Kho,S}
The projective cover of the unique simple $N_n$-module of $L_n$ is $N_n$. Hence
\[ K_0(N_n)= \ZZ[N_n].\]
\end{prop}

Since $N_n$ is a subalgebra of $N_{n+1}$, the algebra $N_{n+1}$ can be considered as an $(N_{n+1}, N_n)$-bimodule or an $(N_n, N_{n+1})$-bimodule. Hence there are induction and restriction functors
\[ Ind_{N_n}^{N_{n+1}}: N_n\text{-mod} \to N_{n+1}\text{-mod} , \qquad M \mapsto \, _{N_{n+1}}N_{{n+1}\, N_n} \otimes_{N_n} M,\]
\[ Res_{N_n}^{N_{n+1}}: N_{n+1}\text{-mod} \to N_n\text{-mod}, \qquad M \mapsto \, _{N_n}N_{{n+1}\, N_{n+1}} \otimes_{N_{n+1}} M.\]

The two functors satisfy the following proposition.

\begin{prop} \cite{Kho,S}
$Ind_{N_n}^{N_{n+1}}$ and $Res_{N_n}^{N_{n+1}}$ are both left and right projective functors for any $n\in \ZZ_{\geq 0}.$ Thus we get the linear maps induced from induction and restriction functors:
\[  Ind_{N_n}^{N_{n+1}}:G_0(N_n)\to G_0(N_{n+1}),\qquad
 Res_{N_n}^{N_{n+1}}:G_0(N_{n+1})\to G_0(N_{n}),\]
\[  Ind_{N_n}^{N_{n+1}}:K_0(N_n)\to K_0(N_{n+1}),\qquad
 Res_{N_n}^{N_{n+1}}:K_0(N_{n+1})\to K_0(N_{n}).\]
\end{prop}

Consider the associative algebra $N=\bigoplus_{n=0}^\infty N_n$ which is not unital but has pairwise orthogonal idempotents $1_{N_n}$, $n\in \ZZ_{\geq 0}.$ Then each $N_n$-module $M$ becomes an $N$-module by letting $aM=0$ for any $a\in N_m$ with $m\neq n.$ Hence a $(N_n, N_m)$-module  is an $(N,N)$-module.

Let
\[\textstyle \mathcal{N}=\bigoplus_{n\geq 0} N_n \text{-mod} \]
be the full subcategory of finite-dimensional $N$-modules. Then the direct sums of functors $Ind=\bigoplus_{n\in \ZZ_{\geq 0}} Ind_{N_n}^{N_{n+1}}$ and $Res=\bigoplus_{n\in \ZZ_{\geq 0}} Res_{N_n}^{N_{n+1}}$ become endofunctors
\begin{equation} \label{Eqn:indres}
\textstyle Ind, Res: \mathcal{N}\to \mathcal{N}.
\end{equation}

Let $G_0(N)$ be the Grothendieck group of $\mathcal{N}$. Then we have  
\[ \textstyle G_0(N)= \bigoplus_{n\geq 0} G_0(N_n).\]
If we denote by $K_0(N)= \bigoplus_{n\geq 0} K_0(N_n)\subset G_0(N)$, then we have the following proposition.

\begin{prop} \cite{Kho,S}
The functors $Ind$, $Res$ induce linear maps between $G_0(N)$ and $K_0(N)$ such that
\[ Res([L_{n+1}])= [L_n], \qquad Ind([L_n])= (n+1) [L_{n+1}], \]
\[ Res([N_{n+1}])= (n+1)[N_n], \qquad Ind([N_n])= [N_{n+1}].\]

\end{prop}

Let us consider $R=\ZZ\text{-span of }\{ x^n/n! | n\geq 0\} $ and  $R'=\ZZ\text{-span of } \{x^n| n\geq 0\}. $ Then both $R$ and $R'$ are closed under the multiplication by $x$ and the derivation $\frac{d}{dx}.$
Consider the linear map
\[ \phi_N:  G_0(N) \to R , \qquad [L_n] \mapsto x^n/n!. \]
Then $\phi_N$ is an isomorphism and $\phi_N|_{K_0(N)}: K_0(N) \to R'$, $[N_n]\mapsto x^n$, is also an isomorphism.

\begin{prop} \label{Prop:2.15_1107} \cite{Kho,S}
Recall that the Weyl algebra is generated by $x$ and $\partial$ with defining relation $\partial x= x\partial +1.$
Then the pair $(\mathcal{N}, \{Ind, Res\})$ gives a categorification of the polynomial representation of the Weyl algebra. More precisely,
\begin{enumerate}
\item the following diagrams are commutative: \\
 \[  \begin{array}[c]{ccc}
\ G_0(N) &\stackrel{Ind, Res}{\xrightarrow{\hspace*{1cm}}} &\ G_0(N)  \\
\downarrow\scriptstyle{{\phi_N}}&&\downarrow\scriptstyle{{\phi_N} }\\
R&\stackrel{x, \partial}{\xrightarrow{\hspace*{1cm}}}&R,
\end{array}
\qquad 
 \begin{array}[c]{ccc}
\ K_0(N) &\stackrel{Ind, Res}{\xrightarrow{\hspace*{1cm}}} &\ K_0(N) \\
\downarrow\scriptstyle{{\phi_N}}&&\downarrow\scriptstyle{{\phi_N}}\\
R'&\stackrel{x, \partial}{\xrightarrow{\hspace*{1cm}}}&R',
\end{array}  \]

\item there exists an isomorphism of endofunctors of $\mathcal{N}$
\begin{equation*}
\begin{aligned}
Res\circ Ind &  \simeq Ind\circ Res \oplus Id.
\end{aligned}
\end{equation*}
\end{enumerate}
\end{prop}

\begin{cor} 
By defining the multiplication on $G_0(N)$ and $K_0(N)$ by
\[ [L_n]\cdot  [L_m]= \frac{(n+m)!}{n!\cdot m!} [L_{m+n}] \quad \text{ and } \quad [N_n]\cdot [N_m]=[N_{n+m}], \quad n,m\in \ZZ_{\geq 0}, \]
we get associative algebra isomorphisms $G_0(N) \simeq \ZZ[x]$ and $K_0(N)\simeq\ZZ[x]$. Moreover, $G_0(N)$ and $K_0(N)$ have differential algebra structures endowed with differentials induced from $Res$.
\end{cor}

\vskip 10mm

\section{ Categorification of Virasoro-Magri PVA} \label{Sec:VMPVA}

Let $\Sigma_n$-mod and $\Sigma_n$-pmod be the category of finitely generated left $\Sigma_n$-modules and the category of finitely generated projective $\Sigma_n$-modules. Let $[M]$ be the isomorphism class of $M$ in $\Sigma_n$-mod or in $\Sigma_n$-pmod.  We denote by $G_0(\Sigma_n)$ and $K_0(\Sigma_n)$ the Grothendieck group of $\Sigma_n$-mod and $\Sigma_n$-pmod. Since $\CC\Sigma_n$ is a semisimple algebra, we have
 \[K_0(\Sigma_n) = G_0(\Sigma_n) \]
and, by Theorem \ref{Thm:Specht}, the Grothendieck group of finite dimensional $\Sigma_n$-module is
\[\textstyle K_0(\Sigma_n)= \bigoplus_{\mu\vdash n }\ZZ[S^\mu]. \]

Let $\Sigma=\bigoplus_{n\in \ZZ_{\geq 0}}\CC\Sigma_n$ be the associative algebra with pairwise orthonormal idempotents $1_{\Sigma_n}$, $n\in \ZZ_{\geq 0}.$ Then any $\Sigma_n$-module becomes a $\Sigma$-module. Consider the full subcategory
\[ \textstyle \mathcal{S}= \bigoplus_{n\in \ZZ_{\geq 0}} \Sigma_n\text{-mod}\]
of the category of finite dimensional $\Sigma$-modules
and its Grothendeick group $K_0(\Sigma).$

Recall that the Virasoro-Magri PVA $\VV$ is isomorphic to the differential algebra of polynomials $\CC[\partial^n L| n\geq 0]$ generated by one variable.  Define the degree on the generators of $\VV$ by 
\[ \deg(L)=\deg(\partial) =1.\]
Then $\VV$ is a graded algebra satisfying
\[ \deg(\partial^n L)= n+1, \quad \deg(L^{m_1} (\partial L)^{m_2} \cdots (\partial^{i-1}L)^{m_i} )= \sum_{k=1}^i k m_k.\]

\begin{lem}\label{Lem:partition}
There is a one-to-one correspondence between degree $n$ monomials in $\VV$ and partitions of $n$.
\end{lem}
\begin{proof}
We can associate $(\partial^{l_1} L)^{m_1} \cdots (\partial^{l_{i}}L)^{m_i}$, $l_1>l_2>\cdots>l_i\geq 0$, to the partition $((l_1+1)^{(m_1)}, (l_2+1)^{(m_2)}, \cdots, (l_i+1)^{(m_i)})$, where $l^{(m)}$ denotes $m$-tuple of $l$'s. Then it gives a one-to-one correspondence between monomials in $\VV$ and partitions.
\end{proof}

\begin{prop}
There is a one-to-one correspondence  between degree $n$ monomials in $\VV$ and irreducible representations of $\Sigma_n$.
\end{prop}
\begin{proof}
It follows from Theorem \ref{Thm:Specht} and Lemma \ref{Lem:partition}.
\end{proof}

Now we would like to find functors $P^{j} Ind : \mathcal{S} \to \mathcal{S}$, $j\geq 1$, and $\bigtriangledown:\mathcal{S} \to \mathcal{S}$ which categorify the multiplication by $\partial^{j-1} L$ and the differential $\partial.$ In order to do this, we introduce the conjugate partition $\mu'\vdash n$ of $\mu\vdash n$ defined as follow:
\[ \mu'=(\mu'_1, \mu'_2, \cdots, \mu'_{l'}) \text{ and } \mu'_i=\#\{\mu_j|\mu_j \geq i\}.\]
In other words, $\mu_i$ is the number of boxes in the $i$-th row in $Y_\mu$ and $\mu'_j$ is the number of boxes in the $j$-th column in $Y_\mu.$

\begin{ex}
If $\mu=(4,2,1)$ then $\mu'=(3,2,1,1).$
\[Y_\mu=
\yng(4,2,1)
\qquad \text{ and } \qquad
Y_{\mu'}=
\yng(3,2,1,1). \]
\end{ex}



Recall that we have the induction functor from $\Sigma_n$-mod to  $\Sigma_{n+1}$-mod. Using the induction functor we construct the endofunctors $P^jInd$ and $\bigtriangledown$ on $\mathcal{S}$.

\begin{lem} Let $\mu\vdash n$ and  $\mu'=(\mu'_1, \cdots, \mu'_l)$ be the conjugate partition of $\mu$. Let  $Y_{\nu^i}$, $i\in \ZZ_{>0}$, be the collection of $n+1$ boxes  which is  arranged in  top-justified columns with $\mu'_j + \delta_{ij}$ boxes in the $j$-th column. Here we assume $\mu_j'=0$ for $j>l$.
\begin{enumerate}
\item Then \[Ind_{n}^{n+1} S^\mu = S^{\nu^1} \oplus \cdots \oplus S^{\nu^{l+1}},\] where
\begin{equation*}
S^{\nu^i} = \left\{
\begin{array}{ll}
\text{ Specht module of shape $\nu^i$ } & \text{ if } i=1 \text{ or }  \mu'_{i-1}> \mu'_i, \\
\qquad 0  & \text{ otherwise. }
\end{array}
\right.
\end{equation*}
\item
Consider the functor  
\[ p_i Ind_n^{n+1} : \Sigma_n\text{-mod} \to \Sigma_{n+1}\text{-mod}\]
for $i\geq 1$ such that  $ S^\mu \mapsto S^{\nu^i}$ if $i=1, \cdots, l+1$ and $S^\mu \mapsto 0$ if $i\geq l+1$.
Then for $n\in \ZZ_{\geq 0}$ and $j\in \ZZ_{>0}$, the functor
\begin{equation} \label{Eqn:P^j}
 P^j Ind_{n}^{n+j}= p_j Ind_{n+j-1}^{n+j}\circ \cdots \circ p_1 Ind_n^{n+1}: \Sigma_n\text{-mod} \to \Sigma_{n+j}\text{-mod},
 \end{equation}
maps  $S^\mu$ to  $S^\nu$, where the conjugate partition of $\nu$ is
\begin{equation*}
 \nu'=\left\{
 \begin{array}{ll}
 (\mu'_1+1, \mu'_2+1, \cdots, \mu'_j+1, \mu'_{j+1}, \cdots, \mu'_l)\vdash n+j &  \text{ if } \quad j\leq l, \\
 (\mu'_1+1, \mu'_2+1, \cdots, \mu'_l+1, 1, \cdots, 1)\vdash n+j & \text{ if } \quad  j>l.
 \end{array}
 \right.
 \end{equation*}
 \end{enumerate}
 \end{lem}

 \begin{proof}
 (1) It follows immediately from Theorem \ref{Thm:branching}. \\
 (2) It can be proved by an induction. Let  $i>1$ and $\nu'=  (\mu'_1+1, \mu'_2+1, \cdots, \mu'_{i-1}+1, \mu'_{i}, \cdots, \mu'_l)$. Since $\mu'$ is a partition, $\nu'$ is a partition satisfying  $\nu'_{i-1}> \nu'_i$. Hence
 \[ p_i Ind_{n+i-1}^{n+i}\circ \cdots \circ p_1 Ind_n^{n+1} S^\mu= p_i Ind_{n+i-1}^{n+i} S^\nu= S^{\nu^+}, \]
 where $\nu$ is the conjugate partition of $\nu'$ and the conjugate partition of $\nu^+$ is $(\mu'_1+1, \mu'_2+1, \cdots, \mu'_{i}+1, \mu'_{i+1}, \cdots, \mu'_l).$
 \end{proof}

\begin{ex}
Let $\mu=(5,2,1)\vdash n=8$ and $j=4$. Then $P^j Ind_{n}^{n+j} S^\mu=S^\nu$, where
\[
Y_\mu=
\yng(5,2,1) \quad \text{ and } \quad
Y_\nu=
\young(\ \ \ \ \ ,\ \ \circ\circ,\ \circ,\circ) .\]
Here the boxes in $Y_\nu$ with circles are the boxes in $Y_\nu \backslash Y_\mu.$
\end{ex}

\begin{prop}
Let 
\begin{equation} \label{Eqn: P^j-}
P^j Ind = \bigoplus_{n\in \ZZ_{>0}}P^j Ind_{n}^{n+j}: \mathcal{S} \to \mathcal{S}.
\end{equation}
Then, for $\mu\vdash n$, we have  $P^j Ind (S^\mu)=S^\nu$, where the Young diagram $Y_\nu$ is obtained by inserting a row with $j$-boxes into  $Y_\mu$.
\end{prop}

\begin{proof}
Let $\mu=(\mu_1, \mu_2, \cdots, \mu_k)$ and $\nu=(\nu_1, \nu_2, \cdots, \nu_{l})$. Then
\[ l=k+1, \quad \nu_i= \mu_i+\text{ min} (\mu_{i-1}-\mu_i, j-\mu_i) \text{ if } \mu_i<j,  \quad \text{ and } \nu_i=\mu_i \text{ if } \mu_i\geq j.\]
In other words,
\[ \nu_{i+1}=\mu_i \text{ if } \mu_i<j, \quad \nu_i= \mu_i \text{ if } \mu_i\geq j\]
and $\nu_i=j$, where $i$ satisfies $\mu_i<j$ and $\mu_{i-1}\geq j.$ Hence $Y_\nu$ can be obtained by adding a row with $j$ boxes to $Y_\mu.$
\end{proof}

We define another functor $\bigtriangledown: \mathcal{S}\to \mathcal{S}$ as follows. Any partition $\mu\vdash n$ can be written as
\[\textstyle \mu= (\mu_1^{(n_1)}, \mu_2^{(n_2)}, \cdots, \mu_k^{(n_k)}), \quad \text{ where }\sum_{i=1}^k n_i \mu_i=n   \]
for $n_i\in \ZZ_{>0}$ and $i=1, \cdots, k$.
Here $\mu_i^{(n_i)}$ denotes  $n_i$-tuple of the number $\mu_i$ and we assume $\mu_i>\mu_{i+1}$  and  $\mu_{k+1}=0.$ Then the induction functor in Theorem \ref{Thm:branching} satisfies
\[ Ind_n^{n+1} S^\mu= S^{\nu^1} \oplus \cdots \oplus S^{\nu^{k+1}},\]
where
\begin{equation*}
\begin{aligned}
& \nu^{i}= (\mu_1^{(n_1)}, \cdots, \mu_{i-1}^{(n_{i-1})}, \mu_i+1, \mu_i^{(n_i-1)}, \mu_{i+1}^{(n_{i+1})}, \cdots, \mu_k^{(n_k)}) \text{  for } i=1, \cdots, k, \\
& \nu^{k+1}=(\mu_1^{(n_1)}, \mu_2^{(n_2)}, \cdots, \mu_k^{(n_k)}, 1).
\end{aligned}
\end{equation*}

For $i\in \ZZ_{>0}$, define the functors 
\[ q_i Ind_{n}^{n+1}: \Sigma_n\text{-mod} \to \Sigma_{n+1}\text{-mod},\qquad S^\mu \mapsto
\left\{ \begin{array}{cc}
S^{\nu^i} &  \text{ if } i=1, \cdots, k, \\
 0 & \text{ otherwise}
 \end{array}\right.
 \]
and 
\begin{equation} \label{Eqn:par}
 \bigtriangledown_n^{n+1}: \Sigma_n\text{-mod} \to  \Sigma_{n+1}\text{-mod}, \qquad S^\mu \mapsto \bigoplus_{i=1}^k\left[ \bigoplus_{j=1}^{n_i}  q_i Ind_n^{n+1} S^\mu \right]= \bigoplus_{i=1}^k (S^{\nu^i})^{n_i}.
 \end{equation}
 Write the direct sum of the functors  by
 \begin{equation} \label{Eqn:par-}
  \bigtriangledown= \bigoplus_{n\in \ZZ_{>0}} \bigtriangledown_n^{n+1}=\mathcal{S}\to \mathcal{S}.
  \end{equation}

Let $\VV_\ZZ= \ZZ[L, \partial L, \partial^2 L, \cdots]$ be the integral form of $\VV$ given in Remark \ref{VM_Z}(2) and let
 \[ \phi_\Sigma:   K_0(\Sigma)\to \VV_\ZZ ,  \quad [S^\mu]\mapsto  (\partial^{l_1} L)^{m_1} \cdots (\partial^{l_{i}}L)^{m_i} ,\]
be the $\ZZ$-linear map where $l_1>l_2>\cdots>l_i\geq 0$ and $\mu= ((l_1+1)^{(m_1)}, (l_2+1)^{(m_2)}, \cdots, (l_i+1)^{(m_i)}).$ Then we have the following theorem.

\begin{thm} \label{Thm:3.3}
Let $P^j Ind, \ \bigtriangledown: K_0(\Sigma)\to K_0(\Sigma)$ be the $\ZZ$-linear maps induced from the endofunctors  $P^j Ind, \ \bigtriangledown:\mathcal{S}\to \mathcal{S}$. Then the following statements are true.
\begin{enumerate}
\item
The following diagrams are commutative:
\begin{equation}
\begin{array}[c]{ccc}
 K_0(\Sigma) &\stackrel{P^j \Ind }{\longrightarrow}& K_0(\Sigma)\\
\downarrow\scriptstyle{\phi_S}&&\downarrow\scriptstyle{\phi_S}\\
\VV_\ZZ&\stackrel{\cdot\, \partial^{j-1}L}{\longrightarrow}&\VV_\ZZ
\end{array} \qquad, 
\qquad
\begin{array}[c]{ccc}
 K_0(\Sigma) &\stackrel{\bigtriangledown}{\longrightarrow}& K_0(\Sigma)\\
\downarrow\scriptstyle{\phi_S}&&\downarrow\scriptstyle{\phi_S}\\
\VV_\ZZ&\stackrel{\partial}{\longrightarrow}&\VV_\ZZ
\end{array}
\end{equation}

\item
We have
\[ \bigtriangledown \circ  P^j Ind = P^{j+1} Ind \oplus P^j Ind \circ \bigtriangledown.\]
\end{enumerate}
\end{thm}

\begin{proof}
(1)  We claim that $\phi_\Sigma \circ P^j  Ind=  \cdot \partial^{j-1}L\circ \phi_\Sigma$. Let $\mu=(\mu_1, \cdots, \mu_l)$ and $\mu_{i-1}>j\geq\mu_i$. Then $P^j Ind [S^\mu]= S^\nu$ where $\nu=(\mu_1, \cdots, \mu_{i-1}, j, \mu_i, \cdots, \mu_l)$ and
\[\phi_\Sigma \circ P^j Ind [S^\mu]= \partial^{\mu_1-1} L \cdots \partial^{\mu_{i-1}-1}L\cdot \partial^{j-1}L \cdot \partial^{i-1}L \cdots \partial^{\mu_l-1}L.\]
On the other hand,
\[\cdot \partial^{j-1}L\circ \phi_\Sigma[S^\mu]=  \partial^{j-1}L\cdot   \partial^{\mu_1-1} L \cdots \partial^{\mu_l-1}L.\]
Since $\VV$ is a commutative algebra, the first diagram commutes.

 For the second commutative diagram, it is enough to check for the irreducible module $S^\mu.$ Let $\mu=(\mu_1^{(n_1)}, \cdots \mu_l^{(n_l)}).$ Then we can see that
\[ \phi_\Sigma \circ \bigtriangledown[S^\mu]=\partial\circ \phi_\Sigma [S^\mu]= \sum_{i=1}^{l}   \left[
n_i\ \partial^{\mu_i} L \cdot  (\partial^{\mu_i-1}L)^{n_i-1} \cdot   \prod_{k=1, \cdots, l, \, k\neq i}  (\partial^{\mu_k-1}L)^{n_k}\right].\]

(2) Let $\mu= (\mu^{(n_1)}, \cdots, \mu_{k}^{(n_{k})}),$ where $\mu_{i-1}>j$ and $\mu_i\leq j.$ Then
\[P^j Ind [S^\mu]= [S^{\overline{\mu}}] \]
where $\overline{\mu}=(\mu^{(n_1)}, \cdots,\mu_{i-1}^{(n_{i-1})}, j, \mu_i^{(n_i)}, \cdots \mu_{k}^{(n_{k})})$ and
\[ \bigtriangledown \circ P^j Ind[S^\mu] = \sum_{l=1}^k n_l [S^{\mu^l}]+[ S^\nu ]\]
where $\mu^l= (\mu_1^{(n_1)}, \cdots, \mu_{i-1}^{(n_{i-1})}, j, \mu_i^{(n_i)}, \cdots, \mu_{l-1}^{(n_{l-1})}, \mu_l+1, \mu_l^{(n_l-1)}, \mu_{l+1}^{(n_{l+1})} \cdots, \mu_k^{(n_k)})$ and $\nu=(\mu^{(n_1)}, \cdots,\mu_{i-1}^{(n_{i-1})}, j+1, \mu_i^{(n_i)}, \cdots \mu_{k}^{(n_{k}})).$\\
On the other hand,
\[ P^j Ind \circ \bigtriangledown[S^\mu]= n_l [S^{\mu^l}], \qquad P^{j+1} Ind [S^\mu]= [S^\nu].\]
Hence we proved the theorem.
\end{proof}

Now we are ready to state and prove the main theorem of this paper.

\begin{thm} \label{Thm:3.8_1106}
Let  $\mu= (\cdots, 3^{(n_3)}, 2^{(n_2)}, 1^{(n_1)})\vdash n$ and $\nu=  (\cdots, 3^{(m_3)}, 2^{(m_2)}, 1^{(m_1)})\vdash m$ where $k^{(n_k)}$ means $k$ appears $n_k$ times in the partition. Define the multiplication on $K_0(\Sigma)$ by 
\[ [S^\mu] \cdot [S^\nu] = [S^{\mu\cup \nu}], \]
where $\mu \cup \nu= (\cdots, 3^{(n_3+m_3)}, 2^{(n_2+m_2)}, 1^{(n_1+m_1)})$.
\begin{enumerate}
\item Then algebra $K_0(\Sigma)$ with the differential $\bigtriangledown$ is isomorphic to the differential algebra $\ZZ[\partial^n L\, |\, n \in \ZZ_{\geq 0}].$
\item Assume that the central charge $c$ is an integer. Then $K_0(\Sigma)$ is endowed with a $\lambda$-bracket and the map $\phi_{\Sigma}$ can be extended to an isomorphism $K(\Sigma)_\CC\to \VV$ as Lie conformal algebras.
\end{enumerate}
\end{thm}

\begin{proof}
By Remark \ref{VM_Z}, $\VV_\ZZ$ has the Virasoro-Magri PVA structure provided that $c$ is an integer. Since each $n$-th bracket, $n\in \ZZ_{\geq 0}$,  can be obtained by multiplications and differentials of generators, we naturally get the $n$-th bracket on $K_0(\Sigma)$ using Theorem \ref{Thm:3.3}.  Consequently, if $K_0(\Sigma)$ and $\VV_\ZZ$ are isomorphic as differential algebras, then $K_0(\Sigma)_\CC$ and $\VV$ are isomorphic as PVAs.

Since we proved that $\phi_\Sigma\circ \bigtriangledown = \partial \circ \phi_\Sigma$, it is enough to show that $\phi_\Sigma$ is an algebra isomorphism. Indeed we have already proved that $\phi_\Sigma$ is bijective and, for $\mu= (\cdots, 3^{(n_3)}, 2^{(n_2)}, 1^{(n_1)})$ and $\nu= (\cdots, 3^{(m_3)}, 2^{(m_2)}, 1^{(m_1)}),$ we have \[\phi_\Sigma([S^\mu]\cdot[S^\nu])= \phi_\Sigma[S^\mu]\cdot \phi_\Sigma[S^\nu]= \prod_{N\in \ZZ_{>0}} (\partial^{N-1} L)^{n_N+m_N}.\] We note that bijectivity of $\phi_\Sigma$ guarantees the Leibniz rule $\bigtriangledown ([S^\mu]\cdot[S^\nu] )= (\bigtriangledown [S^\mu] )\cdot [S^\nu]+ [S^\mu]\cdot  (\bigtriangledown [S^\nu]).$
\end{proof}

\section{ The relation between $K_0(\Sigma)$ and $K_0(N)$} \label{Sec:app}

As we have seen in (\ref{Eqn:Sn}) and (\ref{Eqn:Nn}), the nil-Coxeter algebra is a degenerate homogeneous version of the symmetric group algebra. Thus we can expect $K_0(\Sigma)$ has richer structure than $K_0(N).$ In this section, we show that the Poisson algebra structure and the differential algebra structure on $K_0(N)$ are induced from the PVA structure and the differential algebra structure on $K_0(\Sigma)$.


\subsection{Poisson algebra structure on $K_0(N)$} 
In this section, we focus on the integral form of PVAs and Poisson algebras.   To this end, we assume the central charge $c$ of the Virasoro-Magri PVA $\VV$ is an integer. Also, we define PVA isomorphisms and Poisson algebra isomorphism between integral forms of PVAs and Poisson algebras, respectively.

\begin{defn}\ 
\begin{enumerate}
\item
Let $\VV^1$, $\VV^2$ be PVAs and $\Phi: \VV^1 \to \VV^2$ be a PVA homomorphism. If the $\lambda$-brackets of $\VV^1$ and $\VV^2$ are well-defined on the integral forms $\VV^1_\ZZ$ and $\VV^2_\ZZ$ and 
\[ \Phi|_{\VV^1_\ZZ}: \VV^1_\ZZ \to \VV^2_\ZZ\]
is bijective, then we say $\Phi|_{\VV^1_\ZZ}$ is a {\it PVA isomorphism} between $\VV^1_\ZZ$ and $\VV^2_\ZZ$. 
\item
Let $P^1$, $P^2$ be Poisson algebras and $\Psi: P^1 \to P^2$ be a PVA homomorphism. If the Poisson brackets of $P^1$ and $P^2$ are well-defined on the integral forms $P^1_\ZZ$ and $P^2_\ZZ$ and 
\[ \Psi|_{P^1_\ZZ}: P^1_\ZZ \to P^2_\ZZ\]
is bijective, then we say $\Psi|_{P^1_\ZZ}$ is a {\it Poisson algebra isomorphism} between $P^1_\ZZ$ and $P^2_\ZZ$. 
\end{enumerate}
\end{defn}

Since we have an energy momentum field $L$ in $\VV$, the operator $H=L_{0}$ is Hamiltonian. Hence there is the $H$-twisted Zhu algebra $Zhu_H(\VV)$ which is isomorphic to $\CC[L]$ by Proposition \ref{Prop:zhu}.

In the previous section, we showed that $K_0(\Sigma)$ is isomorphic to $\VV_\ZZ$ as a PVA and $K_0(N)$ is isomorphic to $\ZZ[x]$ as Poisson algebras with trivial Poisson brackets.  Since $\CC\otimes \ZZ[x]\simeq Zhu_H(\VV)$, the Grothendieck group $K_0(\Sigma)$ can be considered as a chiralization of $K_0(N)$. We summarize the above observation as a theorem below.

\begin{thm}
Recall that  $\phi_\Sigma$ is a PVA isomorphism between $K_0(\Sigma)$ and  $\VV_\ZZ$ and $\phi_N$ is a Poisson algebra isomorphism between $K_0(N)$  and $\ZZ[x].$ Let  $P^{j+1}Ind, \bigtriangledown: K_0(\Sigma) \to K_0(\Sigma)$ be the $\ZZ$-linear maps in Theorem \ref{Thm:3.3} and let $Ind: K_0(N)\to K_0(N)$ be the $\ZZ$-linear map in Proposition \ref{Prop:2.15_1107}.
Denote by $\{ [S^\mu]\, _\lambda \, [S^\nu]\}|_{\lambda=0}$ the $0$-th product between $[S^\mu]$ and $[S^\nu]$ in $K_0(\Sigma)$ and denote by $\{ L^\mu\, _\lambda\, L^\nu\}|_{\lambda=0}$ the $0$-th product between $L^\mu=\phi_{\Sigma}([S^\mu])$ and $L^\nu=\phi_{\Sigma}([S^\nu])$ in $\VV_\ZZ$.  Then the following  diagrams are commutative.

 \[ \footnotesize 
  \xymatrix{
& K_0(\Sigma) \ar[rr]^{P^{j+1} Ind} \ar@{.>}'[d][dd]^(.45){Zhu} \ar[ld]_{\phi_\Sigma}&& \makebox[\widthof{$B$}][l] K_0(\Sigma) \ar@{.>}[dd]^-{Zhu} \ar[ld]^{\phi_\Sigma} \\
\VV_\ZZ  \ar[rr]^(.65){\cdot \partial^j L} \ar[dd]_-{Zhu_H} && \VV_\ZZ  \ar[dd]^(.65){Zhu_H} & \\
& K_0(N) \ar[ld]_{\phi_N} \ar'[r][rr]^(0){\delta_{j0} \cdot Ind} && \makebox[\widthof{$B$}][l] K_0(N)
\ar[ld]^{\phi_N} \\ \ZZ[x]  \ar[rr]^{\cdot \delta_{j0} \cdot x} && \ZZ[x]  &
}
 \xymatrix{
& K_0(\Sigma) \ar[rr]^{\bigtriangledown,\ \{[S^\mu]_{\ \lambda}\ \cdot\ \}|_{\lambda=0}} \ar@{.>}'[d][dd]^(.45){Zhu} \ar[ld]_{\phi_\Sigma}&& \makebox[\widthof{$B$}][l] K_0(\Sigma) \ar@{.>}[dd]^-{Zhu} \ar[ld]^{\phi_\Sigma}\\
\VV_\ZZ  \ar[rr]^(.65){ \partial,\ \{L_{ \mu\ \lambda}\  \cdot\ \}|_{\lambda=0}    } \ar[dd]_-{Zhu_H} && \VV_\ZZ  \ar[dd]^(.65){Zhu_H} & \\
& K_0(N) \ar[ld]_{\phi_N} \ar'[r][rr]^(0.3){0 } && \makebox[\widthof{$B$}][l] K_0(N) \ar[ld]^{\phi_N} \\ \ZZ[x]  \ar[rr]^{0 } && \ZZ[x]  &
}
\]
\end{thm}

\begin{proof}
We already showed that the top faces and bottom faces of the diagrams are commutative and  $\CC[x]$ is isomorphic to the $H$-twisted Zhu algebra $Zhu_H(\VV)$ of $\VV$. Since we have  $Zhu_H: \partial^n L \mapsto \delta_{n0} x$ by Proposition \ref{Prop:zhu}, the map $Zhu_H$ induces the $\ZZ$-algebra morphism between $\VV_\ZZ$ and $\ZZ[x]$. Hence it is clear that the left diagram commutes.  Again, by Proposition \ref{Prop:zhu}, the Poisson bracket on $\ZZ[x]$ is induced from the $0$-th product on  $\VV_\ZZ$.  Since the Poisson bracket on  $\ZZ[x]$ is trivial, we get the right commutative diagram.

Also we note that the bottom faces of the diagrams determine the multiplication and the Poisson bracket on $K_0(N)$ as a Zhu algebra of $K_0(\Sigma).$
\end{proof}

\begin{cor}
Let $K_0(\Sigma)_\CC=\CC\otimes_\ZZ K_0(\Sigma)$ and $K_0(N)_\CC=\CC\otimes_\ZZ K_0(N)$. Then $K_0(N)_\CC$ is the Zhu algebra of $K_0(\Sigma)_\CC.$
\end{cor}
\vskip 2mm

\subsection{Differential algebra structure on $K_0(N)$}

The polynomial algebra $\ZZ[x]$ is an differential algebra with the differential $\frac{d}{dx}$.  As a differential algebra
\[ \psi: (\ZZ[x], \frac{d}{dx}) \simeq (\ZZ[L, \partial L, \partial^2 L, \cdots]/(\partial L-1), \partial), \quad x \mapsto L, \]
where $(\partial L-1)$ is the differential algebra ideal generated by $\partial L-1.$ Since $\partial(\partial L-1)=\partial^2L$, we have $\partial^n L=0 $ for $n\geq 2.$  By the composition  of the quotient map $\VV_\ZZ \to  \VV_\ZZ/(\partial L-1)$ and the isomorphism $\psi^{-1}$, we get a surjective algebra homomorphism
$ q:\VV_\ZZ \to \ZZ[x].$ Note that the map $q$ maps $L$ to $x$, $\partial L$ to $1$ and $\partial^n L$ to $0$ for $n\geq 2$. Hence, for $A\in \VV$,
\begin{equation*}
\begin{aligned}
q(\partial^j L \cdot A) = \left\{
\begin{array}{ll}
x \cdot q(A) \qquad \qquad & \text{ if } j=0, \\
q(A) & \text{ if } j=1, \\
0 & \text{ otherwise.}
\end{array}
\right.
\end{aligned}
\end{equation*}
As a consequence, we get the following commutative diagram.

\begin{equation} \label{diagram3}
  \xymatrix{
& K_0(\Sigma) \ar[rr]^{P^{j+1} Ind} \ar@{.>}'[d][dd]^(.45){ } \ar[ld]_{\phi_\Sigma}&& \makebox[\widthof{$B$}][l] K_0(\Sigma) \ar@{.>}[dd]^-{ } \ar[ld]^{\phi_\Sigma}&  \\
\VV_\ZZ  \ar[rr]^(.65){\cdot \partial^j L} \ar@{->>}[dd]_{q} && \VV_\ZZ \ar@{->>}[dd]^(0.65){q} & \\
& K_0(N) \ar[ld]_{\phi_N} \ar'[r][rr]^(0){\delta_{j0} \cdot Ind+\delta_{j1}\cdot id } && \makebox[\widthof{$B$}][l] K_0(N) \ar[ld]^{\phi_N} \\ \ZZ[x]  \ar[rr]^{\cdot \delta_{j0} \cdot x +\delta_{j1} \cdot id } && \ZZ[x]  &
}
\end{equation}

Moreover, since $\partial(L)=1 \ \text{mod}(\partial L-1)$ and $\frac{d}{dx} x=1$ in $\ZZ[x]$ and since both $\partial$ and $\frac{d}{dx}$ satisfy the Leibniz rule, we have the following commutative diagram.
\begin{equation}\label{diagram4}
 \xymatrix{
& K_0(\Sigma) \ar[rr]^{\bigtriangledown} \ar@{.>}'[d][dd]^(.45){ } \ar[ld]_{\phi_\Sigma}&& \makebox[\widthof{$B$}][l] K_0(\Sigma) \ar@{.>}[dd]^-{ } \ar[ld]^{\phi_\Sigma}\\
\VV_\ZZ  \ar[rr]^(.65){ \partial  } \ar@{->>}[dd]_{q } && \VV_\ZZ  \ar@{->>}[dd]^(.65){q } & \\
& K_0(N) \ar[ld]_{\phi_N} \ar'[r][rr]^(0.3){Res } && \makebox[\widthof{$B$}][l] K_0(N) \ar[ld]^{\phi_N} \\ \ZZ[x]  \ar[rr]^{\frac{d}{dx} } && \ZZ[x]  &
}
\end{equation}

By the diagrams  (\ref{diagram3}) and (\ref{diagram4}), we conclude that $K_0(\Sigma)$ induces a differential algebra structure on $K_0(N)$ by the map $\phi_N^{-1} \circ q \circ \phi_\Sigma.$ We remark that this map relates Theorem \ref{Thm:3.3} and Proposition \ref{Prop:2.15_1107} (2).

\subsection{ Quatization of  $K_0(\Sigma)$ via  $K_0(N)$  at $c=0$}
In this section, we assume the central charge $c=0.$
Recall that the state-field correspondence of universal Virasoro vertex algebra $(V(Vir), \mathcal{F})$ is
\begin{equation}\label{Eqn:sf_4}
 s: V(Vir) \to \mathcal{F}, \qquad L_{j_1-2}\cdots L_{-j_s-2}\vac \mapsto :\partial_z^{(j_1)}L(z) \cdots \partial_z^{(j_s)}L(z):.
 \end{equation}
If we take $[L_\lambda L]=(\partial+2\lambda)L$, then  $L_{n}$ and $\vac$ can be identified with the map $x^{-n+1}\frac{d}{dx}$ and the identity map on $\CC[x]$ (See Section \ref{Sec:PVA}).
Thus the map (\ref{Eqn:sf_4}) can be equivalently written as
\[ L_{-j_1-2}\cdots L_{-j_s-2}\vac= :\partial^{(j_1)}L \cdots \partial^{(j_s)}L:\]
and the process of quasi-classical limit induces multiplications of a PVA from normally ordered products of a vertex algebra, i.e., the quasi-claissical limit of $L_{-j_1-2}\cdots L_{-j_s-2}\vac\in V(Vir)$ is $\partial^{(j_1)}L \cdots \partial^{(j_s)}L\in \VV .$  Hence the quantization of each element of $K_0(\Sigma)_\CC$ can be realized as an endomorphism of $K_0(N)_\CC.$ As a consequence, we get the following theorem.

\begin{thm}
Let $\left< Ind, Res \right>$ be the subspace of endomorphisms on $K_0(N)$ generated by $Ind$, $Res$ and their compositions and let $W_\ZZ$ be the $\ZZ$-subalgebra of the Weyl algebra generated by $x$ and $\frac{d}{dx}$. Then the following diagram is commutative,
\begin{equation}
\xymatrix{
K_0(\Sigma) \ar[d]_{\psi_1} \ar[r]^{\phi_\Sigma} &\VV_\ZZ \ar[d]^{\psi_2}\\
\left<Ind, Res \right> \ar[r]^(0.65)i &W_\ZZ}
\end{equation}
where
\begin{equation}
\begin{aligned}
&\psi_1: K_0(\Sigma)_\CC \to  \left<Ind, Res \right>,  \quad \prod_{k=1}^s [S^\mu] \mapsto \prod_{k=1}^s j_k! \, Ind^{j_k+3}  Res, \\
&\psi_2: \VV \to W,  \qquad \prod_{k=1}^s \partial^{j_k}L \mapsto \prod_{k=1}^s j_k! \,  x^{j_k+3}\frac{d}{dx} ,\\
& i (Ind)= x, \qquad i(Res)= \frac{d}{dx},
\end{aligned}
\end{equation}
for $\mu= \phi_\Sigma^{-1}(\prod_{k=1}^s \partial^{j_k}L).$ Hence $\psi_1$ and $\psi_2$ give a realization of  quantizations of $K_0(\Sigma)$ and $\VV_\ZZ,$ respectively.
\end{thm}

\end{document}